\documentclass[12pt, a4paper]{article}
\usepackage[english]{babel}

\usepackage{commath}
\usepackage{amsmath,amssymb}
\usepackage{amsthm}
\usepackage{booktabs}
\usepackage{bookmark}
\usepackage{xcolor}
\usepackage{afterpage}
\usepackage[thinlines]{easytable}
\usepackage{array}
\usepackage{mathtools} 
\usepackage{fancyhdr}
\usepackage{graphicx}
\usepackage{enumitem} 
\usepackage{cite}
\usepackage{mathrsfs}
\usepackage[export]{adjustbox}
\usepackage{setspace}
\usepackage{emptypage}
\usepackage[margin=2.8cm]{geometry} 
\graphicspath{ {./images/} }
\usepackage[utf8]{inputenc}
\usepackage{epigraph}
\usepackage[nottoc,numbib]{tocbibind} 
\usepackage{sectsty}
\usepackage{thmtools}
\usepackage{theoremref}

\usepackage[numbers,sort&compress]{natbib}
\bibpunct[, ]{[}{]}{,}{n}{,}{,}

\theoremstyle{plain}
\newtheorem{theorem}{Theorem}
\newtheorem{lemma}{Lemma}

\newtheorem{proposition}{Proposition}

\theoremstyle{definition}

\newtheorem{example}{Example}
\newtheorem{assumption}{Assumption}

\theoremstyle{remark}

\newtheoremstyle{colon}%
{}
{}
{\itshape}
{}
{\bfseries}
{:}
{ }
{}

\theoremstyle{colon}

\usepackage{setspace}
\providecommand{\keywords}[1]{\textbf{{Keywords: }} #1}

\begin{document}

\title{A Counterexample to Small-time Limit Theorems for Stochastic Processes \footnote{This paper was originally published in Sparago, P. \textit{"A Counterexample to Small-Time Limit Theorems for Stochastic Processes"}, Theory of Probability \& Its Applications, Vol. 71, Iss. 1 (2026) \url{https://doi.org/10.1137/S0040585X97T992847}. This version contains minor corrections to: (i) premise to Lemma \ref{conv_stopping_BM} at the beginning of Section \ref{stopped_paragraph}; (ii) second part of the proof of Lemma \ref{martingale_surface_invariance_langevin_NEW}.}}

\author{{Pietro Maria Sparago$^1$}
}

\date{%
	$^1$London School of Economics\\
	\today
}

\maketitle
	
	\begin{abstract}
		
		The standard small-time functional central limit theorem of semimartingales has been established in \cite{Gerholdetal}, proving that the scaling limit law of a large class of stochastic processes in increasingly small time scales is that of a Brownian motion with a possibly nontrivial variance-covariance matrix. In this paper we focus on the time-homogeneous diffusion processes described by It\^{o} SDEs. Instead of the simple time scaling $1/n$ of \cite{Gerholdetal} we consider the scaled processes stopped at the first exit times from the balls of decreasing radius $n^{-1/2}$ without scaling time itself. To the best of our knowledge, this particular scaling has not been investigated in the literature. We prove that this is a nontrivial example of a sequence of processes which converges in the sense of finite-dimensional distributions over a dense subset of $[0,\infty)$, but it does not converge weakly in the sense of laws of c\`{a}dl\`{a}g processes. We also characterise the limit law of the scaled processes evaluated at their respective first exit times.
		
	\end{abstract}


	\section{Introduction and main results}
	
	Throughout this paper, we will consider $f \in \mathcal{C}^2(\mathbb{R}^d;\mathbb{R}^\ell)$, that is a vector-valued function $f=(f_1,...,f_\ell)$ {such that} each of its components $f_k:\mathbb{R}^d\to \mathbb{R}$ is twice continuously differentiable. The gradient of $f_k$ is denoted with $\nabla f_k$, the Jacobian matrix of $f$ is denoted with $J_f$ and the Hessian matrix of $f_k$ is denoted with $H_{f_k}$. Also, let $(\Omega,\mathscr{F},(\mathscr{F}_t)_{t\geq 0},\mathbf{P})$ be a filtered probability space where the filtration $(\mathscr{F}_t)_{t\geq 0}$ satisfies the usual conditions of being increasing, right-continuous and $\mathbf{P}$-complete. Suppose that there is a standard, $\mathbb{R}^d$-valued $(\mathscr{F}_t)_{t\geq 0}$-Brownian motion $W$ defined on such filtered probability space. We shall denote with $X$ the unique strong solution of the stochastic differential equation
	$$
	X=x+\int_0^.\mu(X_s)ds+\int_0^.\sigma(X_s)dW_s
	$$
	where $x \in \mathbb{R}^d$ and $\mu:\mathbb{R}^d\to \mathbb{R}^d,\,\sigma: \mathbb{R}^d\to \mathbb{R}^{d\times d}$ satisfy the local Lipschitz and linear growth condition. {While the local Lipschitz condition alone yields strong uniqueness of the solution, the linear growth condition guarantees that $X$ is non-explosive (or conservative) so that $\mathbf{P}$-almost surely (a.s.) $X$ has continuous paths; these results can be found, for example, in (\cite{IkedaWatanabe}; Th.3.1 p. 178 and p. 215).} We will work under a diffusivity assumption on the matrix $\sigma$:
	
	\begin{assumption} 
		\label{diffu}
		We assume that
		$$\inf_{y:\|y-x\|\leq 1}(\sigma\sigma')_{k,k}(y)>0$$
		for some $k\in \{1,...,d\}$.
	\end{assumption}
	
	We shall then focus on the sequence of scaled stopped processes
	$$
	(n^{1/2}(f(X_{\tilde{\tau}^n(X)\wedge t})-f(x)))_{t\geq 0},\quad n \in \mathbb{N}
	$$
	where $\tilde{\tau}^n(X):=\inf\{s:\|X_{s}-x\|=n^{-1/2}\}$. Note that time $t$ is not itself scaled. Let $\tau^r(y):=\inf\{t:\sup_{s\leq t}\|y_s\|\geq r\}$ be the first exit time functional. We will prove a convergence result for the finite-dimensional distributions of such sequence. {We will use well-established notation for weak convergence: $\to^\textnormal{w}$ in the case of (laws of) $\mathbb{R}^k$-valued random variables, while the symbol $\Rightarrow$ will be used to indicate weak convergence in the sense of laws of stochastic processes on the path space. We will expand on the latter's framework in the following section.} Our main result is a nontrivial example of a sequence of stochastic processes converging in the finite-dimensional distributions sense, but not weakly - thus lacking tightness.
	
	\begin{theorem} 
		\label{stopped_langevin_fdd_NEW}	
		The sequence of processes $(n^{1/2}(f(X_{\tilde{\tau}^n(X)\wedge .})-f(x)))_{n \in \mathbb{N}}$ converges in the sense of finite-dimensional distributions over any finite collection $\{t_1,t_2,...,t_m\}\subset (0,\infty)$. In particular, for any $\{t_1,t_2,...,t_m\}\subset (0,\infty)$ and $\varphi:\mathbb{R}^{m\times \ell
		}\to \mathbb{R}$ continuous and bounded we have
		$$
		\begin{aligned}
		&|\mathbf{E}[\varphi(n^{1/2}(f(X_{\tilde{\tau}^n(X)\wedge t_1})-f(x)),...,n^{1/2}(f(X_{\tilde{\tau}^n(X)\wedge t_m})-f(x)))]\\
		&-\mathbf{E}[\varphi(J_f(x)\sigma(x)W_{\tau^1(\sigma(x)W)},...,J_f(x)\sigma(x)W_{\tau^1(\sigma(x)W)})]|\,{ \to }\,0
		\end{aligned}
		$$
		{as $n\to \infty$.} However, in general the sequence of processes $(n^{1/2}(f(X_{\tilde{\tau}^n(X)\wedge.})-f(x)))_{n \in \mathbb{N}}$ does not converge weakly to the process $(J_f(x)\sigma(x)W_{\tau^1(\sigma(x)W)})_{t\geq 0}$ - which is a continuous-path process taking the value $J_f(x)\sigma(x)W_{\tau^1(\sigma(x)W)}$ at all $t\geq 0$.
	\end{theorem}
	
	Furthermore, we will also see that if the volatility matrix is a rotation at the initial condition (i.e. $y \mapsto \sigma(y)$ evaluated at $x$ is a rotation matrix) then the limit law of the sequence $(n^{1/2}(f(X_{\tilde{\tau}^n(X)})-f(x)))_{n\in \mathbb{N}}$ - that is, the scaled processes evaluated at the first exit times - depends on a random variable uniformly distributed on the unit sphere. Indeed, using Theorem \ref{stopped_langevin_fdd_NEW} we can prove our second main result.
	
	\begin{theorem}
		\label{th1_NEW}
		We have 
		$$
		n^{1/2}(f(X_{\tilde{\tau}^n(X)})-f(x))\,{\stackrel{\textnormal{w}}{\to}}\, J_f(x)\sigma(x)W_{\tau^1(\sigma(x)W)}
		$$
		{as $n\to \infty$.} In particular, if $\sigma(x)$ is a rotation then for any $\varphi:\mathbb{R}^\ell\to \mathbb{R}$ continuous and bounded we have
		$$
		\mathbf{E}[\varphi(n^{1/2}(f(X_{\tilde{\tau}^n(X)})-f(x)))]\,{{\to}}\, \int_{\mathcal{S}^{d-1}}\varphi(J_f(x)z)\varpi(dz)
		$$
		{as $n\to \infty$} where $\varpi(dz)$ is the normalised canonical surface measure on the unit sphere.
	\end{theorem}
	
	We shall introduce some general lemmas on the weak convergence of stochastic processes that will be useful for the proofs of the central results; these will be presented in the last section respectively in two dedicated subsections. {We remark again that - to the best of our knowledge - the weak convergence of the sequence of stochastic processes of Theorem \ref{stopped_langevin_fdd_NEW} has not been investigated in previous work. It is mentioned in \cite{Stoyanov} that in the literature there are (possibly nontrivial) counterexamples in functional limit theorems, however no specific details are mentioned except general references. We may then provide an explicit illustrative example of Theorem \ref{stopped_langevin_fdd_NEW} as following.}
	
	\begin{example}
		\label{example}
		{Let $d=1,\mu=0,\sigma=1,x=0$, so that $X$ is a standard Brownian motion on $\mathbb{R}$ and $\tilde{\tau}^n(X)=\tilde{\tau}^n(W)=\inf\{s:|W_s|=n^{-1/2}\}$. Suppose for simplicity that $\ell=1$ and $f(0)=0$, so that the sequence of scaled processes is given by
			$$(n^{1/2}f(W_{t\wedge \tilde{\tau}^n(W)}))_{t\geq 0},\quad n\in \mathbb{N}$$
			Then, by Theorem \ref{stopped_langevin_fdd_NEW} such sequence converges in the finite-dimensional distributions sense (over $(0,\infty)$) to the process $(f'(0)W_{\tau^1(W)})_{t\geq 0}$, which takes the (random) value $f'(0)W_{\tau^1(W)}$ at all $t\geq 0$. However, the sequence does not converge weakly in the sense of stochastic processes, in general. Indeed, $n^{1/2}f(W_{0\wedge \tilde{\tau}^n(W)})=0$ for all $n \in \mathbb{N}$, so the coordinate projections at time $t=0$ cannot converge weakly to $f'(0)W_{\tau^1(W)}$ if $f'(0)\neq 0$ since $W_{\tau^1(W)}$ is uniformly distributed on $\{-1,1\}$. A simple, nontrivial choice for $f$ in this case would be $f(x)=e^x-1$, which is null at zero and such that $f'(x)=e^x$, so $f'(0)=1$.}
	\end{example}
	
	\section{Preparatory lemmas}
	
	\subsection{Weak convergence in the local uniform and Skorokhod topologies}
	
	Given a sequence of c\`{a}dl\`{a}g $\mathbb{R}^d$-valued processes $(Y^n)_{n \in \mathbb{N}}$ on a probability space - not necessarily adapted to the same filtration - weak convergence to the law of a c\`{a}dl\`{a}g process $Y$ is denoted with $Y^n\Rightarrow Y$ and it holds iff $\mathbf{E}[\varphi(Y^n)]\to \mathbf{E}[\varphi(Y)]$ as $n \to \infty$ for all functionals $\varphi$ bounded and continuous in the local uniform (resp. Skorokhod) topology, see e.g. (\cite{JacodShiryaev}; VI.1 pp. 325-336) for reference. If $y$ is a continuous function and $(y^n)_{n \in \mathbb{N}}$ a sequence of c\`{a}dl\`{a}g functions, then $y^n\to y$ in the Skorokhod topology iff  $y^n\to y$ in the {local uniform topology. We may abbreviate 'local uniform' with 'l.u.' and 'Skorokhod' with 'Sk.' for ease of exposition. For reference, see} (\cite{JacodShiryaev}; VI.1.17.(b) p. 328). 
	We can then conclude that if the limit process $Y$ has a.s. continuous paths, then weak convergence in each sense is equivalent to the other. 
	
	\begin{lemma} 
		\label{Sk_and_lu_cont}	
		Suppose $\mathbf{P}(Y \in C_{\mathbb{R}^d}[0,\infty))=1$. Let $(Y^n)_{n \in \mathbb{N}}$ be a sequence of c\`{a}dl\`{a}g stochastic processes. Then weak convergence $Y^n\Rightarrow Y$ in the Skorokhod topology is equivalent to weak convergence in the local uniform topology.
	\end{lemma}
	\begin{proof} $(\Rightarrow)$. Suppose we have $Y^n\Rightarrow Y$ in the Skorokhod topology. {Let $\varphi$ be a functional which is bounded and continuous in the l.u. topology}. As $\varphi$ is everywhere l.u. continuous, then $\mathbf{P}(Y \in \{y:\varphi\textrm{ is l.u. continuous at }y\})=1$. But then $1=\mathbf{P}(\{Y\in C_{\mathbb{R}^d}[0,\infty)\}\cap \{Y \in \{y:\varphi\textrm{ is l.u. continuous at }y\}\})\leq \mathbf{P}(Y \in \{y:\varphi\textrm{ is Sk. continuous at }y\})$. So by continuous mapping we have $\varphi(Y^n){\to}^\textnormal{w} \varphi(Y)$ and since $\varphi$ is bounded we get $\mathbf{E}[\varphi(Y^n)]\to \mathbf{E}[\varphi(Y)]$. For $(\Leftarrow)$ we argue equivalently.
	\end{proof}
	
	\subsection{Weak convergence of stopped processes with Brownian limit}
	\label{stopped_paragraph}
	
	We shall denote again with $W$ a standard $\mathbb{R}^d$-Brownian motion. Let $\Sigma$ be a $d\times d$ volatility matrix {such that} $\exists k:(\Sigma\Sigma')_{k,k}>0$. {In particular, this diffusivity assumption on $\Sigma$ can be compared to Assumption \ref{diffu}, while $\sigma(x)$ in Theorem \ref{stopped_langevin_fdd_NEW} is an example of such volatility matrix satisfying the required condition.} We now present the main results on the weak convergence of stopped processes when the limit process is $\Sigma W$.
	
	\begin{lemma} 
		\label{conv_stopping_BM}
		If $Y^n\Rightarrow \Sigma W$, then $\tau^r(Y^n)\to^\textnormal{w} \tau^r(\Sigma W)$.
	\end{lemma}
	\begin{proof} The function $y\mapsto \tau^r(y)$ is continuous at all $y$ {such that} $r\notin\{s:\tau^s(y)<\tau^{s^+}(y)\}$ (where $\tau^{s^+}(y):=\lim_{u\downarrow s}\tau^u(y)$) by (\cite{JacodShiryaev}; VI.2.11 p. 341). The functions $s\mapsto \tau^s(\Sigma W(\omega))$ are nondecreasing left-continuous. For any $\varepsilon>0$ and $s\downarrow r$, by strong Markov property:
		$$
		\begin{aligned}\mathbf{P}(\tau^r(\Sigma W)+\varepsilon<\tau^{r^+}(\Sigma W))&\leq \mathbf{P}(\tau^r(\Sigma W)+\varepsilon<\tau^{s}(\Sigma W))\\
		&=\mathbf{P}\bigg(\sup_{\tau^r(\Sigma W)\leq t \leq \tau^r(\Sigma W)+\varepsilon}\|\Sigma W_t\|<s\bigg)\\
		&=\mathbf{E}\bigg[\mathbf{P}\bigg(\sup_{t\leq \varepsilon}\|z+\Sigma W_t\|<s\bigg)\bigg|_{z=\Sigma W_{\tau^r(\Sigma W)}}\bigg]\\
		&\,{{\to}}\,\mathbf{E}\bigg[\mathbf{P}\bigg(\sup_{t\leq \varepsilon}\|z+\Sigma W_t\|\leq r\bigg)\bigg|_{z=\Sigma W_{\tau^r(\Sigma W)}}\bigg]\\
		&=0
		\end{aligned}
		$$
		{as $s\downarrow r$. The} last equality follows from the fact that if $\|z\|=r$ then {
			$$\mathbf{P}\Big(\sup_{t\leq \varepsilon}\|z+\Sigma W_t\|\leq r\Big)=\mathbf{P}\Big(\sup_{t\leq \varepsilon}\|z+\Sigma W_t\|=r\Big)=0$$
		}As $\mathbf{P}(\tau^r(\Sigma W)+\varepsilon<\tau^{r^+}(\Sigma W))=0$ for all $\varepsilon >0$, we conclude that $\mathbf{P}(\tau^r(\Sigma W)<\tau^{r^+}(\Sigma W))=0$ by measure continuity, and the claim follows by continuous mapping.
	\end{proof}
	
	\begin{lemma} 
		\label{conv_stopped_BM} 
		If $Y^n\Rightarrow \Sigma W$, then $Y^n_{\tau^r(Y^n)\wedge .}\Rightarrow  \Sigma W_{\tau^r(\Sigma W)\wedge .}$.
	\end{lemma}
	\begin{proof} Since $\Sigma W$ has continuous paths and by (the proof of) Lemma \ref{conv_stopping_BM} we have $\mathbf{P}(\tau^r(\Sigma W)<\tau^{r^+}(\Sigma W))=0$, the claim follows from (\cite{JacodShiryaev}; VI.2.12 p. 341) by continuous mapping.
	\end{proof}
	
	\subsection{Processes converging to $0$ in {u.c.p.} and continuous local martingales}
	
	If $(Y^n)_{n \in \mathbb{N}}$ is a sequence of c\`{a}dl\`{a}g stochastic processes and $Y$ a c\`{a}dl\`{a}g limit process, we say that $Y^n\to^{\textrm{ucp}}{Y}$ - uniformly on compact sets in probability - if for any $\varepsilon >0$ and $N \in \mathbb{N}$ we have $\mathbf{P}(\sup_{t\leq N}\|Y^n_t-Y_t\|>\varepsilon)\to 0$ as $n\to \infty$. {As is standard, we abbreviate 'uniformly on compact sets in probability' with 'u.c.p.' for ease of exposition.}
	
	\begin{lemma}[(\cite{JacodShiryaev}; VI.3.31 p. 352)]
		\label{JS_VI_3_31_352}	
		Let $Z^n=H^n+Y^n$. If $H^n\Rightarrow H$ and $Y^n\to^{\textnormal{ucp}}0$, then $Z^n\Rightarrow H$.
	\end{lemma}
	
	\begin{lemma} 
		\label{wconv_to_pconv_constant_processes}	
		Let $c \in C_{\mathbb{R}^d}[0,\infty)$. Then, $Y^n\Rightarrow c$ implies $Y^n\to^\textnormal{ucp}c$.
	\end{lemma}
	\begin{proof} By Lemma \ref{Sk_and_lu_cont} we know that we have weak convergence in the local uniform topology. Let $N \in \mathbb{N}$ and $\varepsilon >0$. 
		Then, {
			$$\varepsilon \mathbf{P}\Big(\sup_{t\leq N}|Y^n_t-c_t|>\varepsilon\Big)\leq \mathbf{E}\Big[\sup_{t\leq N}|Y^n_t-c_t|\wedge \varepsilon\Big]\to  0$$
		}
		as $n\to \infty$. Since $N,\varepsilon$ were arbitrary, this implies the conclusion.
	\end{proof}
	
	If $Y=(Y^1,...,Y^d)$ is a $\mathbb{R}^d$-valued local martingale, we denote with $\langle Y^{k} \rangle $ the predictable quadratic variation of its $k$-th component. A sequence of c\`{a}dl\`{a}g processes $(Y^n)_{n \in \mathbb{N}}$ is called C-tight if it is tight and all possible limit laws are concentrated on the continuous-path processes. {In order to prove Lemma \ref{vanishing_martingales} we shall use results from \cite{JacodShiryaev}, some of which we report here for clarity. Recall that a {local square integrable martingale} is a stochastic process which belongs to the localized class of square integrable martingales (see (\cite{JacodShiryaev}; I.1.33 p. 8)).}
	
	\begin{lemma}[(\cite{JacodShiryaev}; VI.3.33 p. 353)]
		\label{sums_of_c_tight}
		{If $(Y^n)_{n \in \mathbb{N}},(Z^n)_{n \in \mathbb{N}}$ are two C-tight sequences of stochastic processes, then $(Y^n+Z^n)_{n \in \mathbb{N}}$ is C-tight.}
	\end{lemma}
	
	\begin{lemma}[(\cite{JacodShiryaev}; VI.4.13 p. 358)]
		\label{c_tight_martingales}
		{If $(Y^n)_{n \in \mathbb{N}}$ are local square integrable martingales such that $Y^n_0=0$ for all $n$ (i.e. starting at zero), then $(Y^n)_{n \in \mathbb{N}}$ is tight if $G^n:=\sum_{k\leq d}\langle Y^{n,k} \rangle$ is C-tight.}
	\end{lemma}
	
	\begin{lemma} 
		\label{vanishing_martingales}	
		If $(Y^n)_{n \in \mathbb{N}}$ are continuous $\mathbb{R}^d$-valued local martingales starting at $0$ and $\langle Y^{n,k} \rangle_t \to^\mathbf{P} 0$ for all $ k \in \{1,..,d\}, t\geq 0$, then $Y^n\to^\textnormal{ucp} 0$.
	\end{lemma}
	\begin{proof} By (\cite{JacodShiryaev}; I.4.1 p. 38), the processes $Y^n$ are local square integrable martingales. Since $t\mapsto \langle Y^{n,k} \rangle_t$ are nondecreasing processes, $\langle Y^{n,k} \rangle_t \to^\mathbf{P} 0$ for all $t\geq 0$ implies $\langle Y^{n,k} \rangle\to^\textnormal{ucp}0$ 
		- this in turn implies C-tightness of each $\langle Y^{n,k} \rangle$. By { 	 Lemma \ref{sums_of_c_tight}} the sequence of processes given by $G^n:=\sum_{k\leq d}\langle Y^{n,k} \rangle$ is C-tight. Tightness of the original sequence $(Y^n)_{n \in \mathbb{N}}$ then follows from {  Lemma \ref{c_tight_martingales}}. For $k\in \{1,...,d\}$ fixed, since $\langle Y^{n,k} \rangle\to^\textnormal{ucp}0$ the sequence of components $(Y^{n,k})_{n \in \mathbb{N}}$ converges in the finite-dimensional distributions sense to $0$ by (\cite{JacodShiryaev}; VIII.1.9 p. 458). So for each $t\geq 0$ fixed we have $Y^{n,k}_t\to^\mathbf{P}0$ for all $k$ and so $Y^n_t=(Y^{n,1}_t,...,Y^{n,d}_t)\to^\mathbf{P}0$. Then, the convergence of finite-dimensional distributions for $(Y^n)_{n \in \mathbb{N}}$ follows because $(Y^n_{t_1},...,Y^n_{t_m})\to^\mathbf{P}0$ for any $\{t_1,...,t_m\}$. So we have $Y^n\Rightarrow 0$ and we conclude by Lemma \ref{wconv_to_pconv_constant_processes}.
	\end{proof}
	
	\section{Proofs of the main results}
	
	\subsection{Weak convergence in the finite-dimensional distributions sense}
	
	We shall prove Theorem \ref{stopped_langevin_fdd_NEW} at the end of this section. First, we recall the small-time functional central limit theorem of \cite{Gerholdetal} when applied to diffusions described by It\^{o} SDEs. Then, we establish the lemmas for the proof of Theorem \ref{stopped_langevin_fdd_NEW}.
	
	\begin{proposition}[(\cite{Gerholdetal}; Th. 3 p. 725)]
		\label{Gerhold}
		In the sense of weak convergence of stochastic processes, we have $n^{1/2}(f(X_{./n})-f(x))\Rightarrow J_f(x)\sigma(x)W$ as $n \to \infty$ where $W$ is a standard $\mathbb{R}^d$-Brownian motion.
	\end{proposition}
	
	\begin{lemma} 
		\label{local}
		If Assumption \ref{diffu} holds, then for all $t>0$ we have 
		$$\mathbf{P}\bigg(\sup_{s\leq t}\|X_s-x\|=0\bigg)=0
		$$
	\end{lemma}
	\begin{proof} Denote $X_t=(X_{1,t},...,X_{d,t})$ and $\sigma_k$ is the $k$-th row of $\sigma$. We get that {$d[X_k,X_k]_t=\|\sigma_k(X_t)\|^2dt$} and that, for each $k$, the local time process at $x_k$ of the $k$-th component is $L^{x_k}(X_k)$. For each $k$, the process $(a,t)\mapsto L^{a}_{t}(X_k)$ has a modification process which is c\`{a}dl\`{a}g in $a$ and continuous in $t$ a.s. (see e.g. (\cite{RevuzYor}; VI.1.7 p. 225)), and we shall work with such. We then also know that
		$$
		L^{x_k}_{t}(X_k)=\lim_{\varepsilon\to 0}\frac{1}{\varepsilon}\int_0^{t}\mathbf{1}_{[x_k,x_k+\varepsilon)}(X_{k,s})\|\sigma_k(X_s)\|^2ds
		$$
		on a set $\overline{\Omega}$ {such that} $\mathbf{P}(\overline{\Omega})=1$ (see e.g. (\cite{RevuzYor}; VI.1.9 p. 227)). Suppose
		{$$\omega\in\Big\{\sup_{s\leq t}\|X_s-x\|=0\Big\}\cap \overline{\Omega}\cap\{0<\tilde{\tau}^1(X)<\infty\}$$}
		Recall that by Assumption \ref{diffu}, $\mathbf{E}[\tilde{\tau}^1(X)]<\infty$ and that we have $\tilde{\tau}^1(X)>0$ a.s by definition of $\tilde{\tau}^1(X)$, so $\mathbf{P}(0<\tilde{\tau}^1(X)<\infty)=1$. Then, for some $k \in \{1,2,...,d\}$,
		$$
		\begin{aligned}
		L^{x_k}_{t}(X_k)(\omega)&=\lim_{\varepsilon\to 0}\frac{1}{\varepsilon}\int_0^{t}\mathbf{1}_{[x_k,x_k+\varepsilon)}(X_{k,s}(\omega))\|\sigma_k(X_s(\omega))\|^2ds\\
		&\geq \lim_{\varepsilon\to 0}\frac{1}{\varepsilon}\int_0^{t\wedge{\tilde{\tau}^1(X)(\omega)}}\|\sigma_k(X_s(\omega))\|^2ds\\
		&\geq \lim_{\varepsilon\to 0} \Big(\inf_{y:\|y-x\|\leq 1}(\sigma\sigma')_{k,k}(y)\Big)\frac{t\wedge \tilde{\tau}^1(X)(\omega)}{ \varepsilon}=\infty
		\end{aligned}
		$$
		and therefore $a\mapsto L^{a}_{t}(X_k)(\omega)$ cannot be c\`{a}dl\`{a}g at $x_k$. So $\omega \notin \{a\mapsto L^{a}_{t}(X_k)\textrm{ is c\`{a}dl\`{a}g}\}$, the complement of which has probability zero, and we conclude.
	\end{proof}
	
	\begin{lemma} 
		\label{vanish}
		Under Assumption \ref{diffu} we have $\tilde{\tau}^n(X)\downarrow 0$ a.s.
	\end{lemma}
	\begin{proof} Let $\omega \in \{\tilde{\tau}^n(X)\not\to 0\}$. Then there exists $\varepsilon>0$ {such that} for all $N \in \mathbb{N}$ there exists $n\geq N$ {for which} $\tilde{\tau}^n(X)(\omega)>\varepsilon$. Then $\|X_t(\omega)-x\|<n^{-1/2}$ for all $t\leq \varepsilon$, and so $\|X_t(\omega)-x\|=0$ for all $t\leq \varepsilon$. Therefore:
		$$
		\begin{aligned}\mathbf{P}(\tilde{\tau}^n(X)\not\to 0)&\leq \mathbf{P}\bigg(\exists \varepsilon >0:\sup_{t\leq \varepsilon}\|X_t-x\|=0\bigg)\\
		&\leq \sum_{q \in  \mathbb{Q}^+}\mathbf{P}\bigg(\sup_{t\leq q}\|X_t-x\|=0\bigg)
		\end{aligned}
		$$
		and by Lemma \ref{local} we conclude.
	\end{proof}
	
	We remark that we suppose that Assumption \ref{diffu} holds in the following.
	
	\begin{lemma} 
		\label{weak_conv_times}
		We have
		$$
		n\tilde{\tau}^n(X)\,{\stackrel{\textnormal{w}}{\to}}\,\tau^1(\sigma(x)W):=\inf\{t:\|\sigma(x)W_t\|=1\}
		$$
		{as $n\to \infty$.}
	\end{lemma}
	\begin{proof} Note $n\tilde{\tau}^n(X)=\inf\{t:\|n^{1/2}(X_{t/n}-x)\|=1\}$.	Since the paths of $X$ are continuous, the claim follows by Lemma \ref{conv_stopping_BM} and Proposition \ref{Gerhold}.
	\end{proof}
	
	\begin{proposition}
		\label{vanishing_times_jacob_diffusion} 
		We define the sequence of scaled processes
		$$
		\tilde{X}^{n}:=n^{1/2}((f(X_{\tilde{\tau}^n(X)\wedge .})-f(x))-J_f(x)(X_{\tilde{\tau}^n(X)\wedge .}-x))
		$$
		Then $\tilde{X}^{n} \to^\textnormal{ucp} 0$.
	\end{proposition}
	\begin{proof}  Denote $\tilde{X}^n=(\tilde{X}^{n,1},...,\tilde{X}^{n,\ell})$. By It\^{o}
		$$
		\begin{aligned}
		&\tilde{X}^{n,k}_{t}=n^{1/2}\int_0^{\tilde{\tau}^n(X)\wedge t}(\nabla f_k(X_s)-\nabla f_k(x))\cdot \sigma(X_s)dW_s\\
		&+n^{1/2}\int_0^{\tilde{\tau}^n(X)\wedge t}\bigg((\nabla f_k(X_s)-\nabla f_k(x))\cdot \mu(X_s)+\frac{1}{2}\textrm{Tr}[\sigma(X_s)'H_{f_k}(X_s)\sigma(X_s)]\bigg)ds\\
		&=n^{1/2}\int_0^{\tilde{\tau}^n(X)\wedge t}c_k(X_s,x)\cdot dW_s+n^{1/2}\int_0^{\tilde{\tau}^n(X)\wedge t}b_k(X_s,x)ds
		\end{aligned}
		$$
		For any $t\geq 0$ we have by Lemma \ref{weak_conv_times}, Lemma \ref{vanish} and Slutsky's theorem that
		$$
		[\tilde{X}^{n,k},\tilde{X}^{n,k}]_{t}=(n\tilde{\tau}^n(X))\underbrace{\frac{1}{\tilde{\tau}^n(X)}\int_0^{\tilde{\tau}^n(X)\wedge t}\|c_k(X_s,x)\|^2ds}_{\to 0\textrm{ a.s.}}\,{\stackrel{\mathbf{P}}{\to}}\,0
		$$
		{as $n\to \infty$,} because $\|c_k(X_t,x)\|^2\to 0$ a.s. as $t\downarrow 0$. The drift terms converge in {u.c.p.} to zero. Indeed, for any $N \in \mathbb{N}$ we have
		$$
		\begin{aligned}
		&\sup_{t\leq N}\bigg|n^{1/2}\int_{0}^{\tilde{\tau}^n(X)\wedge t}b_k(X_s,x)ds\bigg|\\
		&\leq \frac{n\tilde{\tau}^n(X)}{n^{1/2}}\frac{1}{\tilde{\tau}^n(X)}\int_{0}^{\tilde{\tau}^n(X)\wedge N}|b_k(X_s,x)|ds\,{\stackrel{\mathbf{P}}{\to}}\,0
		\end{aligned}
		$$
		{as $n\to \infty$,} because $|b_k(X_t,x)|\to |\textrm{Tr}[\sigma(x)'H_{f_k}(x)\sigma(x)]|/2$ a.s. as $t\downarrow 0$. Since the sequence of martingale parts of $\tilde{X}^n$ converge in {u.c.p.} to $0$ by Lemma \ref{vanishing_martingales} and the drifts converge in {u.c.p.} to $0$, the claim follows.
	\end{proof}
	
	\begin{lemma}
		\label{fdd_base_langevin}	
		The sequence of processes $(n^{1/2}(X_{\tilde{\tau}^n(X)\wedge.}-x))_{n \in \mathbb{N}}$ converges in the sense of finite-dimensional distributions over $\{t_1,t_2,...,t_m\}\subset (0,\infty)$. In particular, for any $\{t_1,t_2,...,t_m\}\subset (0,\infty)$ and $\varphi:\mathbb{R}^{m\times d}\to \mathbb{R}$ continuous and bounded we have
		$$
		\begin{aligned}
		&|\mathbf{E}[\varphi(n^{1/2}(X_{\tilde{\tau}^n(X)\wedge t_1}-x),...,n^{1/2}(X_{\tilde{\tau}^n(X)\wedge t_m}-x))]\\
		&-\mathbf{E}[\varphi(\sigma(x)W_{\tau^1(\sigma(x)W)},...,\sigma(x)W_{\tau^1(\sigma(x)W)})]|\,{{\to}}\,0
		\end{aligned}
		$$
		{as $n\to \infty$.}
	\end{lemma}
	\begin{proof} For any $0<t_1<...<t_m$ and $\varphi:\mathbb{R}^{m\times d
		}\to \mathbb{R}$ continuous and bounded we have 
		$$
		\begin{aligned}
		&|\mathbf{E}[\varphi(n^{1/2}(X_{\tilde{\tau}^n(X)\wedge t_1}-x),...,n^{1/2}(X_{\tilde{\tau}^n(X)\wedge t_m}-x))]\\
		&-\mathbf{E}[\varphi(n^{1/2}(X_{\tilde{\tau}^n(X)}-x),...,n^{1/2}(X_{\tilde{\tau}^n(X)}-x))]|\\
		&\leq 2\|\varphi\|_\infty \mathbf{P}(\tilde{\tau}^n(X)\geq t_1)\,{{\to}}\,0
		\end{aligned}
		$$
		{as $n\to \infty$,} because $\tilde{\tau}^n(X)\downarrow 0$ a.s. Define $Z^n:=n^{1/2}(X_{./n}-x)$ and recall $\tau^1(Z^n)=n\tilde{\tau}^n(X)$. We have that $Z^n_{\tau^1(Z^n)\wedge .}\Rightarrow \sigma(x)W_{\tau^1(\sigma(x)W)\wedge .}$ by Lemma \ref{conv_stopped_BM} (thus, also in the {finite-dimensional distributions} sense since the limit process has continuous paths). Also note
		$$
		n^{1/2}(X_{\tilde{\tau}^n(X)}-x)=n^{1/2}(X_{(n\tilde{\tau}^n(X))/n}-x)=Z_{\tau^1(Z^n)}^n
		$$
		Denote $Z:=\sigma(x)W$ for ease of notation. Now, for any $a>0$ and $\varphi:\mathbb{R}^d\to \mathbb{R}$ continuous and bounded we have
		$$
		\begin{aligned}&|\mathbf{E}[\varphi(Z_{\tau^1(Z^n)}^n)]-\mathbf{E}[\varphi(Z_{\tau^1(Z)})]|\\
		&\leq 2\|\varphi\|_\infty \mathbf{P}(\tau^1(Z^n)\geq a)+|\mathbf{E}[\varphi(Z_{\tau^1(Z^n)\wedge a}^n)]-\mathbf{E}[\varphi(Z_{\tau^1(Z)\wedge a})]|\\
		&+|\mathbf{E}[\varphi(Z_{\tau^1(Z)\wedge a})]-\mathbf{E}[\varphi(Z_{\tau^1(Z)})]|\\
		\end{aligned}
		$$
		Since $\tau^1(Z^n)\to^\textnormal{w}\tau^1(Z)$ as $n \to \infty$, we get
		$$
		\begin{aligned}\limsup_{n\to \infty}|\mathbf{E}[\varphi(Z_{\tau^1(Z^n)}^n)]-\mathbf{E}[\varphi(Z_{\tau^1(Z)})]|&\leq 2\|\varphi\|_\infty \mathbf{P}(\tau^1(Z)\geq a)\\
		&+|\mathbf{E}[\varphi(Z_{\tau^1(Z)\wedge a})]-\mathbf{E}[\varphi(Z_{\tau^1(Z)})]|
		\end{aligned}
		$$
		Since $\tau^1(Z)$ is a.s. finite and so $Z_{\tau^1(Z)\wedge a}\to Z_{\tau^1(Z)}$ a.s., we conclude by sending $a\to \infty$, and the claim follows.
	\end{proof}
	
	We can now prove Theorem \ref{stopped_langevin_fdd_NEW}.
	
	\begin{proof}[Proof of Theorem \ref{stopped_langevin_fdd_NEW}]
		The first claim follows from Slutsky's theorem by combining Lemma \ref{fdd_base_langevin} and Proposition \ref{vanishing_times_jacob_diffusion}.  For the second claim, suppose the sequence of processes converges weakly to $(J_f(x)\sigma(x)W_{\tau^1(\sigma(x)W)})_{t\geq 0}$. Then, the coordinate projections at $t=0$ would converge in distribution to $J_f(x)\sigma(x)W_{\tau^1(\sigma(x)W)}$. That is, for any $\varphi:\mathbb{R}^\ell\to \mathbb{R}$ we would have
		$$\varphi(0)=\mathbf{E}[\varphi(n^{1/2}(f(X_{\tilde{\tau}^n(X)\wedge 0})-f(x)))]\,{{\to}}\,\mathbf{E}[\varphi(J_f(x)\sigma(x)W_{\tau^1(\sigma(x)W)})]$$
		{as $n\to \infty$,} which would imply that $J_f(x)\sigma(x)W_{\tau^1(\sigma(x)W)}\sim \delta_0$, but this is false in general.
	\end{proof}
	
	\subsection{Convergence in distribution at the first exit time}
	
	We start out with a well-known necessary and sufficient condition for convergence in distribution for $\mathbb{R}^d$-valued random variables; this will allow us to prove a lemma (Lemma \ref{martingale_surface_invariance_langevin_NEW}) instrumental to proving Theorem \ref{th1_NEW}.
	
	\begin{lemma}[(\cite{Klenke}; Th.13.16 p. 254)] 
		\label{weak_lipschitz_portmanteau}	
		Let $(Y^n)_{n \in \mathbb{N}},Y$ be $\mathbb{R}^d$-valued random variables. Then $Y^n\to^\textnormal{w} Y$ iff $\mathbf{E}[\varphi(Y^n)]\to \mathbf{E}[\varphi(Y)]$ for all bounded {Lipschitz continuous} functions $\varphi$.
	\end{lemma}
	
	\begin{lemma}
		\label{martingale_surface_invariance_langevin_NEW}	
		For each $n \in \mathbb{N}$ and $1\leq k \leq \ell$, the processes
		$$
		F^{n,k}_{a}:=n^{1/2}\int_0^{\tilde{\tau}^n(X)\wedge a}\nabla f_k(X_s)\cdot \sigma(X_s)dW_s\,{\stackrel{\textnormal{a.s. and }L^1}{\to}}\, V^{n,k}
		$$
		as $a \to \infty$. We denote $F^n_a=(F^{n,1}_a,...,F^{n,\ell}_a)$ and $V^n=(V^{n,1},...,V^{n,\ell})$. We have
		$$
		V^{n}\,{\stackrel{\textnormal{w}}{\to}}\,  J_f(x)\sigma(x)W_{\tau^1(\sigma(x)W)}
		$$
		{as $n\to \infty$.}
	\end{lemma}
	\begin{proof} For each $n$ and $k$ fixed, the process $(F^{n,k}_{a})_{a \in \mathbb{N}}$ is an $L^2$-martingale. Let $M_{\nabla,k}:=\sup_{\{z:\|z-x\|< 1\}}\|\nabla f_k(z) \sigma(z)\|$. We have by monotone convergence
		$$
		\begin{aligned}
		\sup_{a \in \mathbb{N}}\mathbf{E}[|F^{n,k}_{a}|^2]&=\sup_{a \in \mathbb{N}}n\mathbf{E}\bigg[\int_0^{\tilde{\tau}^n(X)\wedge a}\|\nabla f_k(X_s)\sigma(X_s)\|^2ds\bigg]\\
		&\leq M_{\nabla,k}^2n\mathbf{E}[\tilde{\tau}^n(X)]<\infty
		\end{aligned}
		$$
		where the last inequality follows because Assumption \ref{diffu} holds and $\mathbf{E}[\tilde{\tau}^n(X)]\leq \mathbf{E}[\tilde{\tau}^1(X)]$ for all $n$. So, for each fixed $n \in \mathbb{N},k\in \{1,...,\ell\}$, the martingales $(F^{n,k}_{a})_{a \in \mathbb{N}}$ are uniformly integrable by de la Vall\'{e}e-Poussin criterion (see e.g. (\cite{Schilling_MIM}; Th.22.9(x) p. 266)). Therefore the first claim follows by martingale convergence. Let $\varphi:\mathbb{R}^\ell\to \mathbb{R}$ be bounded and {Lipschitz continuous} with Lipschitz constant $K_\varphi$. By the above it follows that $|\mathbf{E}[\varphi(F_a^n)]-\mathbf{E}[\varphi(V^n)]|\leq K_\varphi \mathbf{E}[\|F_a^n-V^n\|]{\to}0$ as $a\to \infty$. Alternatively, and more straightforwardly, note that $|\mathbf{E}[\varphi(F_a^n)]-\mathbf{E}[\varphi(V^n)]|\leq 2\|\varphi\|_\infty P(\tilde{\tau}^n(X)>a)\leq 2\|\varphi\|_\infty P(\tilde{\tau}^1(X)>a)$ (uniformly in $n$). Let $\varepsilon >0$. For all $a$ sufficiently large:
		$$\begin{aligned}
		&|\mathbf{E}[\varphi(V^n)]-\mathbf{E}[\varphi(J_f(x) \sigma(x)W_{\tau^1(\sigma(x)W)})]|\\
		&\leq \varepsilon + |\mathbf{E}[\varphi(F^n_a)]-\mathbf{E}[\varphi(J_f(x) \sigma(x)W_{\tau^1(\sigma(x)W)})]|
		\end{aligned}
		$$
		Recall that we have that
		$$
		n^{1/2}(f(X_{\tilde{\tau}^n(X)\wedge a})-f(x))=n^{1/2}\int_{0}^{\tilde{\tau}^n(X)\wedge a}\tilde{b}(X_u)du+F^{n}_{a}
		$$
		where $\tilde{b}=(\tilde{b}_1,...,\tilde{b}_\ell)$ similarly as in the notation of (the proof of) Proposition \ref{vanishing_times_jacob_diffusion}. The drift terms converge in probability to $0$ as $n\to \infty$; by Theorem \ref{stopped_langevin_fdd_NEW}  and Slutsky's theorem we have that $F^n_{a}\to^{\textnormal{w}}J_f(x)\sigma(x)W_{\tau^1(\sigma(x)W)}$ as $n\to \infty$. So 
		{$$\limsup_{n\to \infty}|\mathbf{E}[\varphi(V^n)]-\mathbf{E}[\varphi(J_f(x) \sigma(x)W_{\tau^1(\sigma(x)W)})]|\leq \varepsilon$$}
		and since $\varepsilon >0$ was arbitrary, the claim follows by Lemma \ref{weak_lipschitz_portmanteau}.
	\end{proof}
	
	We can now prove our second result, Theorem \ref{th1_NEW}.
	
	\begin{proof}[Proof of Theorem \ref{th1_NEW}] 
		We again use the notation of Lemma \ref{martingale_surface_invariance_langevin_NEW} with $\tilde{b}=(\tilde{b}_1,...,\tilde{b}_\ell)$. By continuous mapping and the first result in Lemma \ref{martingale_surface_invariance_langevin_NEW} we have a.s.:
		$$
		\begin{aligned}
		n^{1/2}(f(X_{\tilde{\tau}^n(X)})-f(x))&=\lim_{a\to \infty}n^{1/2}(f(X_{\tilde{\tau}^n(X)\wedge a})-f(x))\\
		&=n^{1/2}\int_{0}^{\tilde{\tau}^n(X)}\tilde{b}(X_u)du+V^n
		\end{aligned}
		$$
		The first term on the {right-hand side} converges in probability to $0$ as $n\to \infty$. Therefore, by Slutsky's theorem and the second result in Lemma \ref{martingale_surface_invariance_langevin_NEW} we conclude for the first claim. The second claim follows from the fact that if $O$ is a rotation matrix then  by rotational invariance of Brownian motion we have $OW_{\tau^1(OW)} \sim \varpi(dz)$. 
	\end{proof}

\end{document}